\documentclass[11pt]{amsart}

\usepackage{epsfig}  		
\usepackage{epic,eepic}       
\usepackage{amsmath}
\usepackage[retainorgcmds]{IEEEtrantools}
\usepackage{graphicx}
\usepackage[margin=1in]{geometry}
\usepackage{amsfonts}
\usepackage{amsthm}
\usepackage{setspace}
\usepackage{mathtools}
\usepackage{listings}
\usepackage{mathabx}
\usepackage{amssymb}
\usepackage[mathscr]{euscript}
\usepackage{float}
\newtheorem{thm}{Theorem}
\theoremstyle{definition}
\newtheorem{lem}{Lemma}

\newcommand{\im}{\text{im}}

\newcommand{\Z}{\mathbb{Z}}

\newcommand{\R}{\mathbb{R}}

\newcommand{\C}{\mathbb{C}}

\author{Cole Hugelmeyer} 
\title{A solution to the periodic square peg problem}

\begin{document}

\maketitle 

\begin{abstract} We resolve the periodic square peg problem using a simple Lagrangian Floer homology argument.  Inscribed squares are interpreted as intersections between two non-displaceable Lagrangian sub-manifolds of a symplectic 4-torus. \end{abstract}

\section{Introduction}

The Toeplitz square peg conjecture is the long-standing open problem of whether every planar Jordan curve has an inscribed square \cite{survey}\cite{pegs}. Much recent progress has been made on this problem and its variants through the use of symplectic geometry \cite{rect}\cite{quads}\cite{floerpeg}\cite{graphpeg}. In this paper, we utilize this approach to solve the periodic square peg conjecture, discussed by Tao in \cite{tao}. 

We prove the following. 

\begin{thm}[Periodic Square Peg Problem]
Suppose $f$ and $g$ are injective continuous functions $\R\to \R^2$ with disjoint images, satisfying $f(x+1) = f(x) + (0,1)$ and $g(x+1) = g(x) + (0,1)$ for all $x\in \R$. Then there exists a set of four distinct points in $\im(f) \cup \im(g)$ that form the corners of a square in the plane. 
\end{thm}

In particular, we will prove the following result, from which the above theorem is a corollary.

\begin{thm}
Let $f$ and $g$ be smooth embeddings $S^1\to \C/\Z[i]$, both isotopic to the circle $\R/\Z$ and with disjoint images. Then there exist $a_1,a_2,b_1,b_2$ in $S^1$ such that $f(a_2) = f(a_1) + i\cdot(g(b_1) - f(a_1))$ and $g(b_2) = g(b_1) + i\cdot(g(b_1) - f(a_1))$. 
\end{thm}

\section{The symplectic setup}

Let $\omega$ denote the symplectic area form on $\C/\Z[i]$ that is induced by the standard symplectic form on $\C$. We then define a symplectic form on $(\C/\Z[i])^2$ by the formula $\omega_{\pm} = \pi_1^*\omega - \pi_2^*\omega$, where $\pi_1$ and $\pi_2$ are the projections onto the first and second coordinates respectively. 

Now, we let $\tau: ((\C/\Z[i])^2, \omega_{\pm}) \to ((\C/\Z[i])^2, \omega_{\pm})$ be given by the formula $$\tau(a,b) = (a+ i(b-a), b + i(b-a)).$$ We see that this map is a symplectomorphism, because $$\tau^*\omega_{\pm} = |1-i|^2\pi_1^*\omega + |i|^2\pi_2^*\omega - (|1 + i|^2 \pi_2^*\omega + |-i|^2\pi_1^*\omega) = \pi_1^*\omega - \pi_2^*\omega = \omega_{\pm}.$$

Furthermore, if $f$ and $g$ are smooth embeddings $S^1\to \C/\Z[i]$, then we have a Lagrangian sub-manifold $f\times g: S^1\times S^1 \to (\C/\Z[i])^2$. We then see that the set of 4-tuples $(a_1,a_2,b_1,b_2)$ of points in $S^1$ such that  $f(a_2) = f(a_1) + i\cdot(g(b_1) - f(a_1))$ and $g(b_2) = g(b_1) + i\cdot(g(b_1) - f(a_1))$, is in bijective correspondence with the intersection between the Lagrangian sub-manifolds $f\times g$ and $\tau(f\times g)$. We will utilize an abuse of notation where $f\times g$ represents a map, but also the torus it parameterizes.

Viewing the symplectic manifold $((\C/\Z[i])^2, \omega_{\pm})$ as the Cartesian product $\C/\Z[i] \times \overline{\C/\Z[i]}$, we see that we can induce a Hamiltonian isotopy of $f\times g$ by choosing a Hamiltionian isotopy for $f$ and for $g$ within $\C/\Z[i]$. Thus, we see that if $f$ and $g$ have disjoint images, then the Lagrangian sub-manifold $f\times g$ is Hamiltonian isotopic to $f_0\times g_0$, where $f_0(t) = t + \alpha i$ and $g_0(t) = t + \beta i$ for some choice of $\alpha$ and $\beta$ in $\R/\Z$ with $\alpha \neq \beta$. Applying our symplectomorphism $\tau$, we have that $\tau(f\times g)$ is Hamiltonian isotopic to $\tau(f_0\times g_0)$. 

Therefore, to prove Theorem 2, it suffices to prove the following lemma. 

\begin{lem}
Let $f_0(t) = t + \alpha i$ and $g_0(t) = t + \beta i$ for some choice of $\alpha$ and $\beta$ in $\R/\Z$. Then the Lagrangian sub-manifolds $f_0\times g_0$ and $\tau(f_0\times g_0)$ are non-displaceable. 
\end{lem}

\begin{proof}
There is a 4-fold covering map $c: (\C/\Z[i])^2 \to (\C/\Z[i])^2$ given by $c(x,y) = (x+\overline{y}, \overline{x}-y)$. Deck transformations come from translating by $1/2$, $i/2$, or $(1+i)/2$ in both coordinates. Given $f_0(t) = t + \alpha i$ and $g_0(t) = t + \beta i$, we choose $\mu\in \R/\Z$ so that $\mu + \mu = \alpha - \beta$, and we let $\delta = \alpha - \mu$.  Then, we define maps $m,p,q:S^1\to \C/\Z[i]$ given by $m(t) = t + \mu i$ and $p(t) = t - \delta i$ and $q(t) = (1 + 2i)t - \delta i$. Then mapped under $c$, we have that $m\times p$ double covers $f_0\times g_0$, and $m\times q$ double covers $\tau(f_0\times g_0)$. Furthermore, we have that $c^*(\omega_\pm) = 4\omega_\pm$. Since a Hamiltonian isotopy in the base space will induce a Hamiltonian isotopy in the covering space, we see that it suffices to prove that $m\times p$ and $m\times q$ are non-displaceable in $((\C/\Z[i])^2,\omega_\pm)$. In this situation, Floer homology is unobstructed because $\pi_2((\C/\Z[i])^2)$,  $\pi_2((\C/\Z[i])^2, m\times p)$, and $\pi_2((\C/\Z[i])^2, m\times q)$ are all trivial \cite{book}. Furthermore, due to the product structure, we can compute the Lagrangian intersection Floer homology of these sub-manifolds by reducing to a computation of the Lagrangian Floer homology of circles within the constituent 2-tori.  We have $$\dim(HF(m\times p, m\times q; \Lambda)) = \dim(HF(m, m; \Lambda))\cdot \dim( HF(p, q; \Lambda) )= 2*2 = 4.$$ Where $\Lambda$ is the Novikov field. Therefore, $m\times p$ and $m\times q$ are non-displaceable, so $f_0\times g_0$ and $\tau(f_0\times g_0)$ are also non-displaceable. 
\end{proof}

Now that we have Lemma 1, Theorem 2 follows from the prior remarks.

\section{proving the final result}

All that remains is to show that Theorem 2 implies Theorem 1. 

\begin{proof}[Proof of Theorem 1.]
Let $f$ and $g$ be as in the statement of Theorem 1. Let $f_1,f_2,...$ and $g_1,g_2,...$ be sequences of periodic smooth embeddings that limit in the $C^0$ topology to $f$ and $g$ respectively, such that $f_n$ and $g_n$ have disjoint images for all $n$.  Then, let $\lambda$ be such that $[-\lambda,\lambda]\times \R$ contains $\im(f_n)$ and $\im(g_n)$ for all $n$, and let $\varepsilon = \inf_{x\in \R, y\in \R, n\in \Z_{>0}} d(f_n(x), g_n(y))$ be a lower bound for the distance between the images of the $f$ and $g$ functions. Let $N$ be an integer greater than $16 \lambda$. Then, applying Theorem 2 to the pair $\tilde{f}_n(t) = \frac{1}{N}f_n(Nt) + \Z[i]$ and $\tilde{g}_n(t) =  \frac{1}{N}g_n(Nt) + \Z[i]$, we get $a_1,a_2,b_1,b_2$ such that $\tilde{f}_n(a_2) = \tilde{f}_n(a_1) + i\cdot(\tilde{g}_n(b_1) - \tilde{f}_n(a_1))$ and $\tilde{g}_n(b_2) = \tilde{g}_n(b_1) + i\cdot(\tilde{g}_n(b_1) - \tilde{f}_n(a_1))$. Since $\im(\tilde{f}_n)$ and $\im(\tilde{g}_n)$ live in a strip of radius $1/16$ around $\R/\Z$, we can choose planar representatives of $\tilde{f}_n(a_1),\tilde{f}_n(a_2),\tilde{g}_n(b_1),$ and $\tilde{g}_n(b_2)$ that live on $\frac{1}{N}\im(f_n) \cup \frac{1}{N}\im(g_n)$ and are all within distance $1/4$ of each other, and which therefore form a square in the plane. This gives us an inscribed square on $\im(f_n) \cup \im(g_n)$ of side length at least $\varepsilon$. We can also assume that this square lives within $[-N,N]\times[-N,N]$. Thus, we can apply compactness to find a sequence of inscribed squares on $\im(f_n) \cup \im(g_n)$ that converge to an inscribed square on $\im(f) \cup \im(g)$. The lower bound on side length guarantees that the resulting square is nondegenerate. This completes the proof. \end{proof}

\newpage
\bibliography{Refrences}{}
\nocite{*}
\bibliographystyle{plain}

\end{document}